\documentclass[11pt, reqno]{amsart}

\usepackage[english]{babel}
\usepackage[left=3.8cm, right=3.8cm, top=3.5cm, bottom=3.5cm]{geometry}
\usepackage{url, amsfonts, amsmath, amsthm, amssymb, mathtools, mathrsfs, enumerate}
\usepackage{color, xcolor, verbatim, tensor, tikz, tikz-cd}
\usetikzlibrary{angles,quotes}

\usepackage{etoolbox}

\usepackage{hyperref}
\hypersetup{colorlinks=true, linkcolor=blue, citecolor=cyan, urlcolor=magenta, linktoc=all}

\usetikzlibrary{matrix,calc}
\definecolor{wine-stain}{rgb}{0.5,0,0}

\newcommand{\red}[1]{\textcolor{red}{#1}}

\newtheorem{thm}{Theorem}[section]
\newtheorem{prop}[thm]{Proposition}
\newtheorem{lem}[thm]{Lemma}
\newtheorem{cor}[thm]{Corollary}

\theoremstyle{definition}
\newtheorem{defn}[thm]{Definition}
\newtheorem{rem}[thm]{Remark}

\newtheorem*{ques*}{Question}
\newtheorem*{thm*}{Theorem}
\newtheorem*{rem*}{Remark}
\newtheorem*{rems*}{Remarks}
\newtheorem*{exs*}{Examples}
\newtheorem*{mthm*}{Main Theorem}

\numberwithin{equation}{section}

\newcommand{\be}{\beta}

\newcommand{\om}{\omega}

\newcommand{\ul}{\underline}

\renewcommand{\d}{\partial}

\newcommand{\ddbar}{\sqrt{-1}\d\overline{\d}}

\newcommand{\ii}{\sqrt{-1}}

\newcommand{\RR}{\mathbb{R}}

\DeclareMathOperator{\Rc}{Ric}
\DeclareMathOperator{\Cal}{Cal}

\makeatletter
\renewcommand*{\eqref}[1]{%
	\hyperref[{#1}]{\textup{\tagform@{\ref*{#1}}}}%
}
\makeatother

\title[Coupled extremal K\"ahler metrics on ruled surfaces]{On the existence of coupled extremal K\"ahler metrics on ruled surfaces}
\author[Ramesh Mete]{Ramesh Mete}
\address{Department of Mathematics, Indian Institute of Technology Bombay, Powai, Mumbai - 400076, India.}
\email{ramesh2025m@gmail.com, rameshm@math.iitb.ac.in}
\date{\today}

\begin{document}

\begin{abstract}
We investigate the existence conditions for coupled extremal K\"ahler metrics on minimal ruled surfaces over a genus 2 Riemann surface. Using the Calabi ansatz to reduce the coupled extremal equations to ordinary differential equations, we prove that a pair of coupled extremal metrics exists for normalized K\"ahler classes $\Omega_a$ and $\Omega_b$ if and only if their parameters $(a, b)$ belong to an explicitly defined open region $\mathcal{S}_{\mathrm{ext}}$ in the positive real quadrant. Furthermore, we show that this existence region inherently contains the diagonal segment corresponding to classical extremal K\"ahler metrics.
\end{abstract}

\maketitle

\vspace*{2mm}	
\section{Introduction}

A central objective in K\"ahler geometry is the search for canonical metrics, which serve as distinguished representatives within a given K\"ahler class. Among the most prominent examples of such canonical metrics are K\"ahler--Einstein metrics, constant scalar curvature K\"ahler metrics, and Calabi's extremal metrics. Understanding when these metrics exist and whether they are unique is a primary focus of modern complex differential geometry.

\vspace*{1.5mm}
Let $(M,\omega)$ be a compact K\"ahler manifold of complex dimension $n$, with scalar curvature $R(\omega)$ and associated K\"ahler class $[\omega]$. Calabi~\cite{Cal82} introduced several functionals on the space of K\"ahler metrics in $[\omega]$. Among these, the \emph{Calabi functional} is defined by
$$\Cal(\omega):= \int_{M} R(\omega)^2 \omega^n.$$ 
An \emph{extremal K\"ahler metric} is a critical point of the Calabi functional. In particular, he showed that a K\"ahler metric $\omega$ is extremal if and only if its $(1,0)$-gradient vector field $\nabla^{1,0}_{\omega}R(\omega)$ is holomorphic, that is, $$\bar{\partial}\left(\nabla^{1,0}_{\om}R(\omega)\right)=0.$$
Here, for any function $f\in C^\infty(M,\mathbb{C})$, the $(1,0)$-gradient vector field is defined in local coordinates by
$$\nabla^{1,0}_{\omega}f := g^{j\bar{k}}(\d_{\bar k}f)\d_{j},$$ 
where $g$ denotes the Riemannian metric associated with $\omega$. This defines a section of $T^{1,0}M$ and corresponds (up to a factor of $2$) to the $(1,0)$-component of the Riemannian gradient of $f$.

	\vspace*{1.5mm}	
	The notion of an extremal K\"ahler metric generalizes that of a constant scalar curvature K\"ahler (cscK) metric, which in turn generalizes the concept of a K\"ahler–Einstein metric. The Ricci curvature form associated with $(M,\omega)$ is defined by
	$$ \Rc(\omega) := - \ddbar\log(\omega^n),$$
	and the corresponding scalar curvature is given by
	$$R(\omega) := \Lambda_{\omega}\mathrm{Ric}(\omega)= n\,\frac{\mathrm{Ric}(\omega)\wedge\omega^{n-1}}{\omega^n}.$$
	The metric $\omega$ is called K\"ahler--Einstein if $\Rc(\omega) = \lambda\omega$ for some constant $\lambda\in\mathbb{R}$, and it is said to be a cscK metric if its scalar curvature $R(\omega)$ is constant on $M$. The average scalar curvature $\underline{R}$, which depends exclusively on the K\"ahler class $[\omega]$, is expressed as
	$$\ul{R}:= \frac{\int_M R(\omega)\, \omega^n}{\int_M \omega^n} = \frac{2\pi n\, c_1(M)\cdot [\omega]^{n-1}}{[\omega]^n}.$$
	Here, $c_1(M)$ denotes the first Chern class of $M$, defined by $$c_1(M):=\frac{1}{2\pi}[\mathrm{Ric}(\omega)],$$ 
	which is independent of the choice of the K\"ahler metric $\omega$. Observe that $\omega$ is a cscK metric if and only if $R(\omega) = \underline{R}$.

\subsection{Coupled canonical metrics and known results}
\label{sec:coupled-canonical-metrics}
	
	We now introduce the coupled analogues of the canonical K\"ahler metrics discussed above. Generalizing the notion of coupled K\"ahler–Einstein metrics introduced by Hultgren and Witt Nystr\"om~\cite{Hultgren-Witt Nystrom2018}, Datar and Pingali~\cite{Datar-Pingali Coup-cscK} formulated the concept of coupled constant scalar curvature K\"ahler (coupled cscK) metrics. Let $N$ be a non-negative integer. Following \cite{Datar-Pingali Coup-cscK}, an $(N + 1)$-tuple of K\"ahler metrics $(\om_i)_{i=0}^{N}$ is called \emph{coupled cscK} if it satisfies the following system of equations:
	\begin{equation}\label{eq:all Ricci curvature are equal}
		\Rc(\omega_0)=\cdots =\Rc(\omega_{N}),
	\end{equation}
	and
	\begin{equation}\label{eq:coupled cscK second equation}
		R(\omega_0) - \lambda\Lambda_{\omega_0}\om_{\text{sum}} = \underline{R} - \lambda\underline{\omega_{\text{sum}}},
	\end{equation} 
	where $\lambda\in\mathbb{R}$, $\omega_{\text{sum}} := \sum_{i=0}^N \om_i$, and the topological constants $\underline{R}$ and $\underline{\omega_{\text{sum}}}$ are given by
	$$\underline{R} := n \, \frac{2\pi c_1(M)\cdot[\omega_0]^{n-1}}{[\omega_0]^n} \quad \text{and} \quad \underline{\omega_{\text{sum}}} := n \, \frac{[\omega_{\text{sum}}]\cdot[\omega_0]^{n-1}}{[\omega_0]^n}.$$
	Following Stoppa \cite{Stoppa2009tcscK}, when equation \eqref{eq:coupled cscK second equation} is satisfied, $\omega_0$ is said to be a $\theta$-twisted cscK metric with twisting form $\theta := \lambda\omega_{\mathrm{sum}}$.

	\vspace*{1.5mm}
	Observe that condition \eqref{eq:all Ricci curvature are equal} is equivalent to the equality of normalized volume forms:
	\begin{equation}\label{eq:normalized vol elements are equals_coupled canonical metrics}
		\frac{1}{[\omega_0]^n} \omega_0^{n} = \cdots = \frac{1}{[\omega_{N}]^n} \omega_{N}^{n}.
	\end{equation}
	Indeed, from the definition of the Ricci form, $\Rc(\omega_i) = \Rc(\omega_j)$ holds if and only if $\sqrt{-1}\partial\bar{\partial} \log(\omega_i^n / \omega_j^n) = 0$. By the maximum principle on the compact manifold $M$, this implies $\omega_i^n = e^c \omega_j^n$ for some constant $c \in \mathbb{R}$. Integrating over $M$ yields $e^{c} = {[\omega_i]^{n}}/{[\omega_j]^n}$, which proves \eqref{eq:normalized vol elements are equals_coupled canonical metrics}.

	\vspace*{1.5mm}
	Following \cite{Hultgren-Witt Nystrom2018}, an $(N + 1)$-tuple of K\"ahler metrics $(\om_i)_{i=0}^{N}$ is called \emph{coupled K\"ahler-Einstein} if it satisfies
	\begin{equation}
		\label{eq:coupled KE}
		\Rc(\omega_0)=\cdots=\Rc(\omega_N) = \lambda \omega_{\text{sum}}.
	\end{equation}
	We set $\lambda = 1$ if $M$ is Fano (i.e., $c_1(M) > 0$), and $\lambda = -1$ if $c_1(M) < 0$ (i.e., if the canonical bundle $K_M$ is ample). When $2\pi c_{1}(M) = \lambda \sum_{j=0}^{N} [\omega_j]$, any coupled cscK metric reduces to a coupled K\"ahler–Einstein metric. To see this, note that $2\pi c_1(M) = \lambda [\omega_{\mathrm{sum}}]$ implies $\underline{R} = \lambda \underline{\omega_{\mathrm{sum}}}$, so equation \eqref{eq:coupled cscK second equation} simplifies to $R(\omega_0) = \lambda \Lambda_{\omega_0} \omega_{\mathrm{sum}}$. Since both $\Rc(\omega_0)$ and $\lambda \omega_{\mathrm{sum}}$ belong to the class $2\pi c_1(M)$, the $\partial\bar{\partial}$-lemma guarantees the existence of a smooth function $f \in C^\infty(M, \mathbb{R})$ such that $\Rc(\omega_0) = \lambda \omega_{\mathrm{sum}} + \sqrt{-1}\partial\bar{\partial} f$. Taking the trace with respect to $\omega_0$ yields $$R(\omega_0) = \lambda \Lambda_{\omega_0} \omega_{\mathrm{sum}} + \Delta_{\omega_0} f,$$ which implies that $\Delta_{\omega_0} f = 0$. By compactness of $M$, $f$ is constant, establishing $\Rc(\omega_0) = \lambda \omega_{\mathrm{sum}}$. Combined with \eqref{eq:all Ricci curvature are equal}, this recovers \eqref{eq:coupled KE}.

	\vspace*{1.5mm}	
	For $N=0$, equation \eqref{eq:coupled KE} reduces to the classical K\"ahler--Einstein condition for $\omega_0$. In general, a tuple $(\omega_i)_{i=0}^{N}$ is coupled K\"ahler--Einstein if and only if, for each $j = 0, \cdots, N$, the component metric $\omega_j$ is a $(\lambda \sum_{i \neq j} \omega_i)$-twisted K\"ahler--Einstein metric. Furthermore, as in \cite{Hultgren-Witt Nystrom2018}, any standard K\"ahler--Einstein metric $\omega_{\mathrm{KE}} \in 2\pi\lambda c_1(M)$ induces a \emph{trivial} coupled K\"ahler–Einstein metric: for any choice of positive constants $\lambda_i > 0$ satisfying $\sum_{i=0}^{N} \lambda_i = 1$, the tuple $(\lambda_i \omega_{\mathrm{KE}})_{i=0}^{N}$ satisfies \eqref{eq:coupled KE} because $$\Rc(\lambda_i \omega_{\mathrm{KE}}) = \Rc(\omega_{\mathrm{KE}}) = \lambda \omega_{\mathrm{KE}}.$$ Observe that existence of coupled K\"ahler--Einstein metrics gives a decomposition of the K\"ahler class $2\pi\lambda c_1(M)$. More precisely, if $\alpha_i$ denotes the K\"ahler class of $\omega_i$, then \eqref{eq:coupled KE} yields
	$$\sum_{i =0}^{N} \alpha_i = 2\pi\lambda c_1(M).$$

	\vspace*{1.5mm}
	Hultgren \cite{Hultgren2019} provided an example of a \emph{toric} \emph{Fano} four-manifold $M$ that does not admit a K\"ahler--Einstein metric (because the Futaki invariant of $M$ is nonzero), but admits a coupled K\"ahler--Einstein metric for a decomposition $(\alpha_1, \alpha_2)$ of $2\pi c_1(M)$. Moreover, as noted in \cite[Remark 6]{Hultgren2019}, one can construct a decomposition $(\alpha_1, \alpha_2)$ of $2\pi c_1(M)$ even on a Fano K\"ahler--Einstein manifold $M$ such that there are no coupled K\"ahler--Einstein metrics in $(\alpha_1, \alpha_2)$. When $c_1(M) < 0$, it was established by Hultgren and Witt Nystr\"om \cite{Hultgren-Witt Nystrom2018} via the calculus of variations, and by Pingali \cite{Pingali2018cKE} via the continuity method, that any decomposition of $-2\pi c_1(M)$ admits a unique coupled K\"ahler--Einstein metric. If $M$ is Fano and $(\omega_i)_{i=0}^N$ is a coupled K\"ahler--Einstein metric on a given decomposition $(\alpha_0, \cdots, \alpha_N)$ of $2\pi c_1(M)$, then $\mathrm{Aut}_0(M)$ is reductive; furthermore, if $(\omega'_i)_{i=0}^N$ is another coupled K\"ahler--Einstein metric on $(\alpha_0, \cdots, \alpha_N)$, then $\omega'_i = f^{\ast}\omega_i$ for some $f \in \mathrm{Aut}_0(M)$ (cf. \cite{Hultgren-Witt Nystrom2018}). Datar and Pingali \cite{Datar-Pingali Coup-cscK} gave a moment map interpretation for the coupled cscK system (and hence for the coupled K\"ahler--Einstein equation). In recent years, these coupled canonical metrics have been extensively studied (see, e.g., \cite{Datar-Pingali Coup-cscK, DelHult2021, FutZhang2021cKE, FutZhang2021cSasRic-sol, Hashimoto2023-balanced, Hultgren2019, Hultgren-Witt Nystrom2018, Lee2023-moment-coupled, Mete2026ccscK, Nak2021cKE, Nak2023cKE, Pingali2018cKE, Tak2021cKE-geom-qua, Tak2021cKE-Ric-ite} and references therein).

	\vspace*{1.5mm}
	As noted above, extremal K\"ahler metrics generalize cscK metrics by relaxing the requirement of constant scalar curvature to the condition that their $(1,0)$-gradient vector fields are holomorphic. Analogously, we extend the coupled cscK condition by requiring that the coupled scalar expression generates a holomorphic vector field with respect to $\omega_0$.
	
	\begin{defn}\label{def:coupled-extremal-metrics}
		An $(N + 1)$-tuple of K\"ahler metrics $(\omega_i)_{i=0}^{N}$ is called \emph{coupled extremal} if the condition \eqref{eq:all Ricci curvature are equal} holds and the following holomorphicity condition is satisfied:
		\begin{equation}\label{eq:coupled extremal second equation}
			\bar{\partial} \left( \nabla^{1,0}_{\omega_0} \left( R(\omega_0) - \Lambda_{\omega_0}\omega_{\text{sum}} \right)\right) =0.
		\end{equation}
	\end{defn}

	\vspace*{1.5mm}
	For simplicity, we have assumed that $\lambda = 1$. When $N=0$, this definition reduces to that of classical extremal metrics. For $N \geq 1$, observe that if $[\omega_i] = [\omega_0]$ for all $i \in \{1, \dots, N\}$, then the tuple $(\omega_i)_{i=0}^{N}$ with $\omega_i = \omega_0$ is a coupled extremal metric if and only if $\omega_0$ is itself an extremal metric.

\subsection{Main results}
	
	In this paper, we study the existence of coupled extremal metrics on certain compact complex surfaces, $X := \mathbb{P}(L\oplus\mathcal{O})$, known as (minimal) ruled surfaces. Here, $L$ is a degree $-1$ holomorphic line bundle over a genus $2$ Riemann surface $(\Sigma, \omega_{\Sigma})$ equipped with a K\"ahler metric $\omega_{\Sigma}$, and $\mathcal{O}$ denotes the trivial line bundle over $\Sigma$. These complex surfaces are explicit examples of pseudo-Hirzebruch surfaces studied in \cite{Ton98-extrem}. It is well known that any K\"ahler class on the ruled surface $X$ takes the form $\Omega := a_1 \mathrm{C} + a_2 D_{\infty}$ for constants $a_1, a_2 > 0$, where $\mathrm{C}$ represents the Poincar\'e dual of a fiber of $X$ and $D_{\infty}$ is the infinity divisor on $X$. Throughout this work, we focus on normalized K\"ahler classes of the form $\Omega_{a} := 2\pi(\mathrm{C} + a D_{\infty})$ for $a > 0$.

	\vspace*{1.5mm}
	In \cite[Theorem 1.3]{Mete2026ccscK}, it was shown that any pair of normalized K\"ahler classes $(\Omega_a, \Omega_b)$ with $a, b > 0$ on the ruled surface $X$ fails to admit a coupled cscK metric, owing to the non-vanishing of the generalized (or coupled) Futaki invariant introduced in \cite{Datar-Pingali Coup-cscK}. In contrast, there exists a pair of K\"ahler metrics $(\omega, \chi) \in \Omega_a \times \Omega_b$ such that $\omega$ is a $\chi$-twisted cscK metric whenever $b \geq \frac{a(5a + 4)}{2(a + 1)}$ (see \cite[Theorem 1.1]{Mete2026ccscK}).

	\vspace*{1.5mm}
	Our main result concerning the existence of coupled extremal K\"ahler metrics on $X$ is as follows.
	
	\begin{thm}\label{thm:coupled-extremal-metric-exis-1}
		Let $X = \mathbb{P}(L\oplus\mathcal{O})$ be a minimal ruled surface over a genus $2$ Riemann surface $\Sigma$. For normalized K\"ahler classes $\Omega_a := 2\pi(\mathrm{C} + a D_\infty)$ and $\Omega_b := 2\pi(\mathrm{C} + b D_\infty)$ with $a, b > 0$, there exists a pair of coupled extremal K\"ahler metrics $(\omega, \chi)\in\Omega_a \times \Omega_b$ satisfying the Calabi ansatz (see Section~\ref{subsec:momentum-profile}) if and only if $(a, b) \in \mathcal{S}_{\mathrm{ext}}$, where $\mathcal{S}_{\mathrm{ext}}\subset\mathbb{R}_{>0}^2$ is an open region containing the diagonal segment
		\begin{equation}\label{eq:diagonal-segment-std-extremal-metric}
			\Delta_{k_1} := \Big\{(x, x)\in\mathbb{R}_{>0}^2:~ 0 < x < k_1\Big\}.
		\end{equation}
		Here, $k_1 \approx 18.889$ is the unique positive root of the quartic polynomial 
		\begin{equation}\label{eq:quartic-polynomial-std-extremal-metric}
			\wp(x) := x^4 - 16x^3 - 52x^2 - 48x -12.
		\end{equation}
	\end{thm}

	\vspace*{1.5mm}
	The diagonal segment $\Delta_{k_1}$ is in one-to-one correspondence with the set of normalized K\"ahler classes $\{\Omega_x :\, x \in (0, k_1)\}$ that admit an extremal K\"ahler metric. It is expected that the set $\mathcal{S}_{\mathrm{ext}}$ is also connected.

	\vspace*{1.5mm}
	As a corollary of Theorem~\ref{thm:coupled-extremal-metric-exis-1}, we have the following non-existence result.
	
	\begin{cor}\label{cor:non-exis-coup-extrem-small-value-of-b}
		Let $X = \mathbb{P}(L\oplus\mathcal{O})$ be a minimal ruled surface over a genus $2$ Riemann surface $\Sigma$. Consider the normalized K\"ahler classes $\Omega_a := 2\pi(\mathrm{C} + a D_\infty)$ and $\Omega_b := 2\pi(\mathrm{C} + b D_\infty)$ on $X$, where $a, b > 0$. For every $a > k_0$, there exists a $\delta_{a} > 0$ such that for all $b \in (0, \delta_a)$, the pair $(\Omega_a, \Omega_b)$ does not admit a coupled extremal K\"ahler metric satisfying the Calabi ansatz. Here,  $k_0 \approx 12.451$ is the unique positive root of the polynomial
		\begin{equation}\label{eq:quartic-polynomial-cor-non-exis}
			\vartheta(x) := 9x^4-88x^3-280x^2-240x-48.
		\end{equation}
	\end{cor}

	\vspace*{1.5mm}
	We have the following result regarding the boundedness of the set $\mathcal{S}_{\mathrm{ext}}$.
	
	\begin{prop}\label{prop:bddness-of-the-coupled-extremal-region}
		For every fixed $x > 0$, the set of values of $y > 0$ for which $(x, y) \in \mathcal{S}_{\mathrm{ext}}$ is bounded above.
	\end{prop}

	\vspace*{1.5mm}
	Combining Theorem~\ref{thm:coupled-extremal-metric-exis-1} and Proposition~\ref{prop:bddness-of-the-coupled-extremal-region}, we obtain the following non-existence result.
	
	\begin{cor}\label{cor:non-exis-coup-extrem-large-value-of-b}
		Let $X = \mathbb{P}(L\oplus\mathcal{O})$ be a minimal ruled surface over a genus $2$ Riemann surface $\Sigma$. Consider the normalized K\"ahler classes $\Omega_a := 2\pi(\mathrm{C} + a D_\infty)$ and $\Omega_b := 2\pi(\mathrm{C} + b D_\infty)$ on $X$, where $a, b > 0$. For every fixed $a > 0$, there exists a $K(a) > 0$ such that for all $b > K(a)$, the pair $(\Omega_a, \Omega_b)$ does not admit a coupled extremal K\"ahler metric satisfying the Calabi ansatz.
	\end{cor}

	\begin{rem}
		For $a > k_0$, we expect that $K(a) > \delta_a > 0$ and that $k_1$ lies in the open interval $(\delta_a, K(a))$.
	\end{rem}

	\vspace*{1.5mm}
	The strategy for proving Theorem \ref{thm:coupled-extremal-metric-exis-1} is divided into two steps:
	
	\begin{itemize}
		\item \textbf{Step 1:} Under the Calabi ansatz, the system of coupled extremal K\"ahler metric equations \eqref{eq:all Ricci curvature are equal}--\eqref{eq:coupled cscK second equation}---specifically, \eqref{eq:Ricci curv equal, N=2}--\eqref{eq:coupled extremal second equ, N=2}---reduces to a pair of ordinary differential equations (ODEs). These ODEs govern a momentum profile $\phi(\tau)$ associated with $\omega$ and a function $\psi(\tau)$ related to the momentum profile of $\chi$, where $\tau \in [0, a]$. Applying the boundary conditions \eqref{eq:boundary conditions for psi}, $\psi$ is uniquely determined as in \eqref{eq:unique-positive-solution-psi}:
		$$\psi(\tau) = -1 + \sqrt{\frac{b^2+2b}{a^2+2a}(\tau^2+2\tau)+1}.$$
		Substituting $\psi(\tau)$ into the second-order ODE \eqref{eq:ODE coupled extremal, N=2}, integrating twice, and applying the boundary conditions \eqref{eq:boundary conditions for momentum profile} and \eqref{eq:boundary conditions for psi} yields
		$$F(\tau) = P_4(\tau) + \frac{a^2+2a}{3(b^2+2b)}\left(1 - \left(\frac{b^2+2b}{a^2+2a}(\tau^2+2\tau)+1\right)^{3/2}\right),$$
		where $F(\tau) = (1+\tau)\phi(\tau)$ and $P_4(\tau)$ is a polynomial of degree four (see \eqref{eq:solution-function-F} for its explicit formula). By construction, $(\omega,\chi)$ form a pair of coupled extremal K\"ahler metrics if and only if the function $\phi(\tau) = \frac{F(\tau)}{1+\tau}$ defines a valid momentum profile on $(0, a)$, which is equivalent to requiring that $F(\tau) > 0$ for all $\tau \in (0, a)$.
		
		\vspace*{1.5mm}
		\item \textbf{Step 2:} Given the pair $(\Omega_a, \Omega_b)$ of normalized K\"ahler classes on $X$, we obtain a function $F(\tau)$ given by \eqref{eq:solution-function-F} that satisfies the boundary conditions \eqref{eq:bdd-cond-for-solution-function-F}. To emphasize its dependence on the parameters $a$ and $b$, we write $F(a, b, \tau)$ instead of $F(\tau)$. Consider the existence region
		$$\mathcal{S}_{\mathrm{ext}} = \Big\{(x, y)\in \mathbb{R}_{>0}^2 :~ F(x, y, \tau) > 0 \text{ for all } \tau \in (0, x)\Big\}.$$
		Here, $F(x, y, \cdot)$ is the solution function for the pair $(\Omega_x, \Omega_y)$. Note that $(\omega,\chi)\in\Omega_a \times \Omega_b$ form a pair of coupled extremal K\"ahler metrics if and only if the parameter point $(a, b)$ belongs to $\mathcal{S}_{\mathrm{ext}}$.
	\end{itemize}

	\vspace*{1.5mm}
	We show that $\mathcal{S}_{\mathrm{ext}}$ is an open subset of $\mathbb{R}_{>0}^2$ containing the diagonal segment $\Delta_{k_1}$. Note that if $(a, b)$ lies on the boundary $\partial \mathcal{S}_{\mathrm{ext}}$, then the metric $\omega$ degenerates.

	\vspace*{1.5mm}
	Consider the positive constant
	\begin{align*}
		k_{+} := \frac{k_1 - 2 + \sqrt{k_{1}^2 + 6k_1 + 4}}{5} \approx 7.733.
	\end{align*}
	Note that $0< a < k_+$ if and only if $a < \frac{a(5a + 4)}{2(a + 1)} < k_1$. Moreover, observe that $k_+ < k_0 < k_1$. Fix any $a\in(0, k_{+})$ and any $b\in \left(\frac{a(5a + 4)}{2(a + 1)}, k_1\right)$. It is an interesting question whether the pair of normalized K\"ahler classes $(\Omega_a, \Omega_b)$ admits a coupled extremal K\"ahler metric. More generally, one can extend our results to ruled surfaces over a Riemann surface of genus $\textbf{g} \geq 2$ with a holomorphic line bundle $L$ of degree $- m$, or investigate the existence of coupled extremal metrics for general $N$-tuples on higher-dimensional ruled manifolds. We want to explore these directions in future work.
	
	\medskip
	
	\noindent\textbf{Organization of the paper:} The remainder of the paper is organized as follows. Section~\ref{sec:preliminaries} contains the necessary preliminaries, including the Calabi ansatz, the momentum profile, and the relevant intersection-theoretic computations. Section~\ref{sec:exis-coupled-extrem-ruled-surf} establishes the existence of coupled extremal K\"ahler metrics on the ruled surface and proves Theorem \ref{thm:coupled-extremal-metric-exis-1}. Finally, Section~\ref{sec:further-analysis-of-coupled-extrem-region} further analyzes the structure of the coupled extremal existence region and proves Corollary~\ref{cor:non-exis-coup-extrem-small-value-of-b} and Proposition~\ref{prop:bddness-of-the-coupled-extremal-region}.

	\vspace*{2mm}
	\section{Preliminaries}
	\label{sec:preliminaries}
	
	Let $(M,\omega)$ be a compact K\"ahler manifold of complex dimension $n$. Let $R_{i\bar j}$ denote the components of the Ricci curvature form $\mathrm{Ric}(\omega)$, and let $R^{\bar j i}$ be obtained by raising indices using the metric $g$ corresponding to the K\"ahler form $\omega$. For a variation $\omega_t := \omega + \ddbar f$ within the fixed K\"ahler class $[\omega]$, the first variation of the Calabi functional is given by (cf. \cite{Sze-bk})
	\begin{align*}
		\frac{d}{dt}\Big|_{t=0}\Cal(\om_t) &= -2 \int_M f \left(\Delta_{\omega}R(\omega) + \nabla_{j}\left(R^{\bar{k} j}\nabla_{\bar k}R(\omega)\right)\right) \omega^n \\ &= -2 \int_M f \mathcal{L}_{\omega}R(\omega)\, \omega^n.
	\end{align*}
	Here, $\mathcal{L}_{\omega} := \Delta_{\omega} + \nabla_{j}(R^{\bar{k} j}\nabla_{\bar k}\cdot)$ denotes the fourth-order Lichnerowicz operator (where $\Delta_\omega$ is the complex Laplacian). Consequently, the K\"ahler metric $\omega$ is a critical point of the Calabi functional if and only if the $(1,0)$-gradient vector field $\nabla_{\omega}^{1,0}R(\omega)$ is holomorphic.

	\vspace*{1.5mm}
	Using the average scalar curvature $\underline{R}$ of the K\"ahler class $[\omega]$, we observe that
	$$\Cal(\om) = \int_M (R(\omega) - \underline{R})^2\, \omega^n  + \underline{R}^2\int_M \omega^n.$$
	It is straightforward to see that cscK metrics minimize the Calabi functional. Furthermore, by the work of Donaldson \cite{Don2005} and Chen \cite{Chen2009}, extremal K\"ahler metrics are also global minimizers of the Calabi functional in their K\"ahler class. In particular, the scalar curvature of an extremal K\"ahler metric is constant if and only if the Futaki invariant \cite{Fut88-bk} vanishes on the given K\"ahler class.

	\subsection{Ruled surfaces over Riemann surfaces of genus 2}
	\label{subsec:ruled-surf-over-genus-2-Riemm-surf}
	
	Let $(\Sigma, \omega_\Sigma)$ be a genus $2$ Riemann surface equipped with a K\"ahler metric $\omega_\Sigma$ whose scalar curvature satisfies $R(\omega_\Sigma) = -2$. Since the Euler characteristic of $\Sigma$ is $\chi(\Sigma) = -2$, by the Gauss--Bonnet theorem we have
	$$\text{Area}(\Sigma, \omega_\Sigma) := \int_{\Sigma} \omega_\Sigma = 2\pi.$$
	Let $(L, h)$ be a degree $-1$ holomorphic line bundle over $\Sigma$ equipped with a Hermitian metric $h$ whose Chern curvature form $\gamma(h) := -\sqrt{-1}\partial\bar{\partial} \log h$ equals $-\omega_\Sigma$. In particular, $c_1(L) < 0$. By definition, the degree of $L$ being $-1$ means that $\int_{\Sigma} c_1(L) = -1$, which implies $\int_\Sigma \gamma(h) = -2\pi$ for any Hermitian metric $h$ on $L$.\footnote{Recall that $c_1(L) = \frac{1}{2\pi}[\gamma(h)]$.} Then the K\"ahler surface 
	$$X := \mathbb{P}(L \oplus \mathcal{O})$$
	is called a (minimal) \emph{ruled surface}. Note that $X$ is a $\mathbb{P}^1$-bundle over $\Sigma$ and serves as an example of a \emph{pseudo-Hirzebruch surface} (cf.~\cite{Ton98-extrem}).

	\vspace*{1.5mm}
	Let $\mathrm{C}$ denote the Poincar\'e dual of a fiber of $X$ (so that $\mathrm{C}$ is a copy of the Riemann sphere $\mathbb{S}^2$ embedded in $X$). Let $D_\infty$ (resp. $D_0$) be the image of the subbundle $L \oplus \{0\} \subset L \oplus \mathcal{O}$ (resp. $\{0\} \oplus \mathcal{O} \subset L \oplus \mathcal{O}$) under the bundle projectivization map to $X$. They are called the \emph{infinity divisor} and \emph{zero divisor} of $X$, respectively. Both $D_\infty$ and $D_0$ are isomorphic to $\Sigma$ inside $X$ (where $\Sigma$ is identified with $D_0$ as a complex curve in $X$). The following intersection formulas hold (see \cite{BHPdeVen-bk, Sze-bk, Ton98-extrem}):
\begin{equation}\label{eq:intersection formluae for ruled surface with genus 2}
\begin{split}
& \mathrm{C}^2 =0, \quad  D_{\infty}^2 =1, \quad D_{0}^2 = D_{0}\cdot [\Sigma] = -1, \\ 
& \mathrm{C}\cdot D_{\infty} = \mathrm{C}\cdot D_{0} = \mathrm{C}\cdot [\Sigma] = 1, \quad D_{\infty}\cdot D_{0} = D_{\infty}\cdot[\Sigma] = 0,\\ 
& c_1(L)\cdot [\Sigma] = -1, \quad [\om_\Sigma]\cdot [\Sigma] = 2\pi,
\end{split}
\end{equation}
where $[\Sigma] \in H^2(\Sigma, \RR)$ is the {\em fundamental class} of $\Sigma$ (identified with $D_0$ in $H^2(X, \mathbb{R})$). The last two formulas follow from $\deg(L)=-1$ and $\text{Area}(\Sigma,\om_\Sigma)=2\pi$, respectively. Since $\Sigma$ is a Riemann surface, we have $H^2(\Sigma,\RR) = H^{1,1}(\Sigma,\RR)$.

	\vspace*{1.5mm}
	By the Leray--Hirsch theorem (see \cite{Hatcher-bk}), we have 
	$$ H^2(X,\RR) = \RR \mathrm{C}\oplus \RR D_{\infty}.$$ 
	The results of \cite{Fuj92, LeBrun-Singer93, Ton98-extrem} imply that a class $\be \in H^{1,1}(X,\RR)$ is K\"ahler if and only if 
	$$\be^2 >0, \quad \be \cdot \mathrm{C} >0, \quad \be\cdot D_\infty >0, \quad \text{and} \quad \be \cdot D_0 >0.$$ 
	In particular, the \emph{K\"ahler cone} (i.e. the set of all K\"ahler classes) of $X$ is given by
	$$ \mathcal{K}(X) = \Big\{ a\mathrm{C} + b D_{\infty} \in H^{1,1}(X,\RR):~ a > 0,~ b >0\Big\}.$$

	\vspace*{1mm}
	\subsection{Calabi ansatz and momentum profile}
	\label{subsec:momentum-profile}
	
	Consider a normalized K\"ahler class $\Omega_a := 2\pi(\rm{C} + a D_\infty)$, where $a>0$. We want to construct a K\"ahler form $\om$ in this K\"ahler class satisfying the Calabi ansatz (cf. \cite{Cal82}).

	\vspace*{1.5mm}
	On $X\setminus(D_\infty \cup D_0)$ we consider the logarithm of the fiber-wise norm: 
	$$s = \ln |(z,w)|_{h}^2 = \ln h(z) + \ln |w|^2$$ where $z$ is a local coordinate on $\Sigma$ and $w$ is a fiber coordinate on $L$, and $h$ denotes the Hermitian metric on $L$ mentioned before. Let $p : L \to \Sigma$ denotes the projection map and consider the following $(1,1)$-form (cf. \cite{Cal82, Sze-bk}):
	\begin{equation}\label{eq:Calabi-ansatz-express-for-omega}
		\omega = p^\ast\omega_\Sigma + \ddbar u(s),
	\end{equation}
	where $u \in C^\infty(\mathbb{R}, \mathbb{R})$. If $u$ is a \emph{strictly convex} function such that $s \mapsto s+u(s)$ is strictly increasing, then $\omega$ is a K\"ahler form on $X\setminus(D_\infty \cup D_0)$. To extend $\omega$ to the entire space $X$ and ensure it belongs to the Kähler class $\Omega_a$, the following asymptotic conditions must be satisfied:
	\begin{equation}\label{eq:aymptotic-limit-for-derivative-of-u}
		\lim_{s \to -\infty} u'(s) = 0 \quad \text{and} \quad \lim_{s \to \infty} u'(s) = a.
	\end{equation}

	\vspace*{1.5mm}
	We make the change of coordinates $\tau := u'(s) \in [0,a]$ and set $\phi(\tau) := u''(s)$. The asymptotic conditions for $u'$ imply (see \cite{Sze-bk}):
	\begin{equation}\label{eq:boundary conditions for momentum profile}
		\phi(0) = \phi(a) = 0 \quad \text{and} \quad \phi'(0) = - \phi'(a) = 1.
	\end{equation}
	Note also that $\phi > 0$ on $(0,a)$ because $u$ is strictly convex. In general, any smooth function $\phi : [0,a] \to \mathbb{R}$ that is strictly positive on $(0,a)$ and satisfies the boundary conditions in \eqref{eq:boundary conditions for momentum profile} is called a \emph{momentum profile}. Thus, to every metric $\omega$ satisfying Calabi symmetry, there corresponds a momentum profile.

	\vspace*{1.5mm}
	As in~\cite{Sze-bk}, for a suitable choice of local trivialization of $L$, the K\"ahler metric $\omega$ is given locally by
	\begin{equation}\label{eq:local-expression-for-omega}
		\begin{split}	
			\omega = (1+\tau) p^\ast\omega_\Sigma + \phi(\tau) \frac{\sqrt{-1} dw \wedge d\bar{w}}{|w|^2}.
		\end{split}
	\end{equation}
	The volume form with respect to $\omega$ is then given by
	\begin{equation}\label{eq:vol-form-express-ruled-surface}
		\frac{\omega^2}{2} = (1+\tau)\phi(\tau)p^\ast\omega_\Sigma \wedge \frac{\ii dw \wedge d\bar{w}}{|w|^2} 
		= (1+\tau) p^\ast\omega_\Sigma \wedge d\tau \wedge d\theta.
	\end{equation}
	Here, the latter equality follows from the identity
	$$\phi(\tau)\frac{\ii dw \wedge d\bar{w}}{|w|^2} = d\tau \wedge d\theta,$$ 
	where $\theta = \arg(w)$ in polar coordinates. Indeed, writing $w = r e^{\sqrt{-1}\theta}$, we have $s = \ln \vert w\vert^2 = 2\ln r$, which yields: 
	\begin{equation*}
		\frac{\sqrt{-1} dw \wedge d\bar{w}}{|w|^2} = \frac{2}{r} dr \wedge d\theta = ds \wedge d\theta.
	\end{equation*}
	The identity then follows directly from $d\tau = \phi(\tau) ds$.

	\vspace*{1.5mm}
	Using \eqref{eq:vol-form-express-ruled-surface}, the Ricci curvature of $\omega$ is given by
	\begin{equation}\label{eq:Ric-formula-in-terms-of-momentum-profile}
		\Rc(\omega) = \left( -2 - \frac{[(1+\tau)\phi]'}{1+\tau} \right) p^\ast\omega_\Sigma  -  \left( \frac{[(1+\tau)\phi]'}{1+\tau} \right)' \phi(\tau)\frac{\ii dw \wedge d\bar{w}}{|w|^2}.
	\end{equation}
	Applying the relation $R(\omega)\omega^n = n \Rc(\omega) \wedge \omega^{n-1}$ for $n=2$, the scalar curvature of $\omega$ simplifies to
	\begin{equation}\label{eq:scal-curv-formula-in-terms-of-momentum-profile}
		R(\omega) = -\frac{2}{1+\tau} - \frac{[(1+\tau)\phi]''}{1+\tau}.
	\end{equation}

	\vspace*{1.5mm}
	We consider K\"ahler metrics $\omega\in\Omega_a := 2\pi(\mathrm{C} + a D_\infty)$ and $\chi\in\Omega_b := 2\pi(\mathrm{C} + b D_\infty)$, both satisfying the Calabi ansatz. The metric $\omega$ is locally given by \eqref{eq:local-expression-for-omega} in terms of a momentum profile $\phi$ satisfying the boundary conditions \eqref{eq:boundary conditions for momentum profile}. Similarly, since $\chi$ also satisfies the Calabi ansatz, we have
	$$ \chi = p^\ast\omega_\Sigma + \ddbar v(s)$$ for some strictly convex function $v\in C^\infty(\RR,\RR)$ such that $s \mapsto s+v(s)$ is strictly increasing. The function $v$ satisfies the following asymptotic conditions:
	$$\lim_{s \to -\infty} v'(s) = 0 \quad \text{and} \quad \lim_{s \to \infty} v'(s) = b. $$
	Furthermore, for a suitable local trivialization of the holomorphic line bundle $L$, the metric $\chi$ is locally given by (cf.~\cite{Sze-bk})
	\begin{equation}
		\label{eq:local-express-chi}
		\chi = (1+ v') p^\ast\omega_\Sigma + v'' \frac{\sqrt{-1} dw \wedge d\bar{w}}{|w|^2}.
	\end{equation}

	\vspace*{1.5mm}
	We fix the variable $\tau := u'(s) \in (0,a)$ corresponding to the metric $\omega$. Let $\psi : [0,a] \to [0,b]$ be the smooth function defined by $\psi(\tau) := v'(s)$, satisfying the boundary conditions
	\begin{equation}\label{eq:boundary conditions for psi}
		\psi(0)=0 \quad \text{and} \quad \psi(a)=b.
	\end{equation}
	Note that $\psi$ is a bijection. Differentiating $\psi(\tau) = v'(s)$ with respect to $s$ yields
	$$\psi'(\tau)\phi(\tau) = v''(s).$$ 
	Thus, the function $[0,b] \ni \sigma \mapsto \psi'(\tau)\phi(\tau)$, where $\psi(\tau)=\sigma$, is the momentum profile for the K\"ahler metric $\chi$. Consequently, from \eqref{eq:local-express-chi}, we obtain the local expression for $\chi$ in terms of its momentum profile:
	\begin{equation}\label{eq:local-expression-for-chi-momentum-profile}
		\chi =  (1+\psi(\tau)) p^\ast\omega_\Sigma + \psi'(\tau)\phi(\tau) \frac{\sqrt{-1} dw \wedge d\bar{w}}{|w|^2}.
	\end{equation}
	In particular, the volume form with respect to $\chi$ is given by
	\begin{equation}\label{eq:vol-form-wrt-chi-on-ruled-surf}
		\frac{\chi^2}{2} = (1+\psi(\tau))\psi'(\tau)\phi(\tau) p^\ast\omega_\Sigma \wedge \frac{\sqrt{-1} dw \wedge d\bar{w}}{|w|^2}.
	\end{equation}
	Moreover, from \eqref{eq:local-expression-for-omega} and \eqref{eq:local-expression-for-chi-momentum-profile}, the eigenvalues of $\chi$ with respect to $\omega$ (i.e., of the endomorphism $\omega^{-1}\chi$) are $\frac{1 + \psi(\tau)}{1+\tau}$ and $\psi'(\tau)$. Therefore, the trace of $\chi$ with respect to $\omega$ is given by
	\begin{equation}\label{eq:trace-func-of-chi-wrt-omega}
		\Lambda_{\omega} \chi = \frac{1 + \psi(\tau)}{1+\tau} + \psi'(\tau).
	\end{equation}

	\subsection{Some intersection number computations}
	
	Recall that $\omega\in\Omega_a := 2\pi(\mathrm{C} + a D_{\infty})$ and $\chi\in\Omega_b := 2\pi(\mathrm{C} + b D_{\infty})$, where $a, b > 0$. Using \eqref{eq:intersection formluae for ruled surface with genus 2}, we obtain
	\begin{equation}\label{eq:values-of-intersection-prod-normalized-Kah-class}
		\begin{split}
			[\om]^2 &= 4\pi^2 (\mathrm{C} + a D_\infty)^2 = 4\pi^2 (a^2 + 2a), \\
			[\chi]^2 &= 4\pi^2 (\mathrm{C} + b D_\infty)^2 = 4\pi^2 (b^2 + 2b), \\
			[\chi]\cdot[\om] &= 2\pi (\mathrm{C} + b D_\infty)\cdot 2\pi (\mathrm{C} + a D_\infty) = 4\pi^2 (b+a+ab).
		\end{split}
	\end{equation}
	Note that the same values can be obtained by directly integrating the local expressions \eqref{eq:local-expression-for-omega} and \eqref{eq:local-express-chi} for K\"ahler metrics $\om$ and $\chi$, respectively. Consequently, this yields
	\begin{align*}
		\ul{\chi} := 2 \frac{[\chi]\cdot[\om]}{[\om]^2} = 2 \frac{b+a+ab}{a^2+2a}.
	\end{align*}

	\vspace*{1.5mm}
	As shown in~\cite{Mete2026ccscK}, the first Chern class of the ruled surface $X = \mathbb{P}(L \oplus \mathcal{O})$ is given by $c_1(X) = -3\mathrm{C} + 2 D_\infty$, which can be derived by integrating the Ricci formula \eqref{eq:Ric-formula-in-terms-of-momentum-profile} alongside the boundary conditions \eqref{eq:boundary conditions for momentum profile} and applying the intersection formulas \eqref{eq:intersection formluae for ruled surface with genus 2}. Using \eqref{eq:intersection formluae for ruled surface with genus 2} once more, we compute
	\begin{equation*}
		2\pi c_1(X)\cdot [\om] = 2\pi (-3\mathrm{C} + 2 D_\infty) \cdot 2\pi (\mathrm{C} + a D_\infty) = 4\pi^2 (2-a).
	\end{equation*}
	In particular, the average scalar curvature is given by
	\begin{align}\label{eq:average-scal-curv-value-for-normalized-Kah-class}
		\ul{R} := 2 \frac{2\pi c_1(X)\cdot[\om]}{[\om]^2} = \frac{2(2-a)}{a^2 + 2a}.
	\end{align}

	\vspace*{2mm}
	\section{Existence of coupled extremal K\"ahler metrics}
	\label{sec:exis-coupled-extrem-ruled-surf}
	
	In this section, we establish the existence of ($2$-tuple) coupled extremal metrics on the ruled surface $X = \mathbb{P}(L\oplus \mathcal{O})$. Fix two normalized K\"ahler classes $\Omega_a := 2\pi (\mathrm{C} + a D_{\infty})$ and $\Omega_b := 2\pi (\mathrm{C} + b D_{\infty})$ with $a, b > 0$ on $X$, as before. Recall that a pair of K\"ahler metrics $(\omega,\chi)\in\Omega_{a}\times\Omega_{b}$ is called a coupled extremal K\"ahler metric if $\Rc(\omega) = \Rc(\chi)$, which is equivalent to
	\begin{equation}\label{eq:Ricci curv equal, N=2}
		\frac{\omega^2}{[\omega]^2} = \frac{\chi^2}{[\chi]^2},
	\end{equation}
	and additionally satisfies the holomorphicity condition
	\begin{equation}\label{eq:coupled extremal second equ, N=2}
		\bar{\partial} \left(\nabla^{1,0}_{\omega} (R(\omega) - \Lambda_{\omega}\chi )\right) = 0.
	\end{equation}
	Observe that when the K\"ahler classes $\Omega_a$ and $\Omega_b$ coincide, a metric $\omega$ is extremal if and only if the pair $(\omega, \omega)$ is a coupled extremal K\"ahler pair.

	\subsection{Reduction to ODEs and explicit solutions}
	
	Suppose both $\om$ and $\chi$ satisfy the Calabi ansatz. Then, using the volume form expressions \eqref{eq:vol-form-express-ruled-surface} and \eqref{eq:vol-form-wrt-chi-on-ruled-surf}, the equation \eqref{eq:Ricci curv equal, N=2} can be written as
	\begin{equation*}
		\frac{(1+\tau)\phi(\tau)}{4\pi^2(a^2 + 2a)} = \frac{(1+\psi(\tau))\psi'(\tau)\phi(\tau)}{4\pi^2(b^2 + 2b)},
	\end{equation*}
	which simplifies to the following first-order ordinary differential equation (ODE)
	\begin{equation}\label{eq:ODE for all Ricci equal, N=2}
		(1+\psi(\tau)) \psi'(\tau) = \frac{b^2 + 2b}{a^2 + 2a} (1+\tau).
	\end{equation}
	The general solution to \eqref{eq:ODE for all Ricci equal, N=2} takes the form
	\begin{equation*}
		(1+\psi(\tau))^2 = \frac{b^2 + 2b}{a^2 + 2a} (1+\tau)^2 + C_1,
	\end{equation*}
	where $C_1$ is a constant of integration. Applying the boundary condition $\psi(0) = 0$ from \eqref{eq:boundary conditions for psi} yields
	$$ C_1 = \frac{(a-b)(a+b+2)}{a^2 + 2a}.$$
	Therefore, the unique positive solution to the ODE \eqref{eq:ODE for all Ricci equal, N=2} satisfying the boundary conditions \eqref{eq:boundary conditions for psi} is given by
	\begin{equation}\label{eq:unique-positive-solution-psi}
		\psi(\tau) = -1 + \sqrt{\frac{b^2 + 2b}{a^2 + 2a}(\tau^2 + 2\tau) + 1}.
	\end{equation}

	\vspace*{1.5mm}
	Next, from \eqref{eq:scal-curv-formula-in-terms-of-momentum-profile} and \eqref{eq:trace-func-of-chi-wrt-omega}, we observe that the function $R(\omega) - \Lambda_{\omega}\chi$ depends solely on the variable $\tau$. More precisely,
	\begin{align*}
		R(\om) - \Lambda_{\om}\chi 
		= -\frac{ 2 + [(1+\tau)\phi(\tau)]''}{1+\tau} - \frac{[(1+\psi(\tau))(1+\tau)]'}{1+\tau},
	\end{align*}
	where the primes denote derivatives with respect to $\tau$. For any function $f(\tau)$ of the variable $\tau$, we have (cf.~\cite[p.~74]{Sze-bk})
	$$\nabla^{1,0}_{\om}f(\tau)=f'(\tau)w\frac{\d}{\d w},$$ 
	where $w$ is a fiber coordinate on the line bundle $L$. Then the holomorphic condition $\bar{\partial} (\nabla^{1,0}_{\omega}f(\tau)) = 0$ implies that $f'(\tau)$ must be constant; that is, $f(\tau) = A \tau + B$ for some constants $A, B \in \mathbb{R}$. Consequently, equation \eqref{eq:coupled extremal second equ, N=2} reduces to the following second-order ODE:
	\begin{equation}\label{eq:ODE coupled extremal, N=2}
		[(1+\tau)\phi(\tau)]'' + [(1+\psi(\tau))(1+\tau)]' = -(A\tau^2 + (A+B)\tau + B + 2),
	\end{equation}
	where $A, B \in \mathbb{R}$ are constants whose explicit values will be determined later.

	\vspace*{1.5mm}
	Integrating the ODE \eqref{eq:ODE coupled extremal, N=2}, applying the boundary conditions \eqref{eq:boundary conditions for momentum profile} and \eqref{eq:boundary conditions for psi}, and using \eqref{eq:unique-positive-solution-psi} yields
	\begin{equation}\label{eq:solution-of-ODE-for-holomorphicity-1st-derivative-of-phi}
		\begin{split}
			&[(1+\tau)\phi(\tau)]' + (1 + \tau)\sqrt{\frac{b^2 + 2b}{a^2 + 2a}(\tau^2 + 2\tau) + 1}  \\
			&\qquad = -A\frac{\tau^3}{3} - (A+B)\frac{\tau^2}{2} - (B + 2)\tau + 2. 
		\end{split}
	\end{equation}
	Integrating once more and applying \eqref{eq:boundary conditions for momentum profile} again, we find that the solution to \eqref{eq:ODE coupled extremal, N=2} is given by
	\begin{equation}\label{eq:solution-of-ODE-for-holomorphicity-involving-A-B}
		\begin{split}
			&(1 +\tau)\phi(\tau) + \frac{a^2 + 2a}{3(b^2 + 2b)}\left(\frac{b^2 + 2b}{a^2 + 2a}(\tau^2 + 2\tau) + 1 \right)^{\frac{3}{2}} \\ 
			&\qquad = -A\frac{\tau^4}{12} - (A+B)\frac{\tau^3}{6} - (B + 2)\frac{\tau^2}{2} + 2\tau + \frac{a^2 + 2a}{3(b^2 + 2b)}.
		\end{split}
	\end{equation}

	\vspace*{1.5mm}
	Now we determine the values of the constants $A$ and $B$. Setting $\tau = a$ in \eqref{eq:solution-of-ODE-for-holomorphicity-1st-derivative-of-phi} and applying the boundary conditions \eqref{eq:boundary conditions for momentum profile}, we obtain
	\begin{equation}\label{eq:solving-constants-A-and-B-1st-equ}
		A \frac{a^2(2a+3)}{6} + B \frac{a(a+2)}{2} = 2 - 2a - b(1+a).
	\end{equation}
	On the other hand, setting $\tau = a$ in \eqref{eq:solution-of-ODE-for-holomorphicity-involving-A-B} and simplifying yields
	\begin{equation}\label{eq:solving-constants-A-and-B-2nd-equ}
		A \frac{a^2(a+2)}{12} + B \frac{a(a + 3)}{6} = \frac{-(a+2)b^2 - 6ab + 6 - 9a}{3(b+2)}.
	\end{equation}
	Solving the system of linear equations \eqref{eq:solving-constants-A-and-B-1st-equ} and \eqref{eq:solving-constants-A-and-B-2nd-equ} for $A$ and $B$, we find
	\begin{equation}\label{eq:explicit-value-of-constants-A-and-B}
		\begin{split}
			A &= \frac{12(b^2 + 2a^2 b + 5a^2 + 4a)}{a^2(a^2 + 6a + 6)(b+2)}, \\
			B &= \frac{- 2(a^2 + 5a + 6)b^2 - 24a(a+1)b + 12(- 4a^2 - 3a + 2)}{a(a^2 + 6a + 6) (b+2)}.
		\end{split}
	\end{equation}
	Therefore, from \eqref{eq:solution-of-ODE-for-holomorphicity-involving-A-B}, the unique solution satisfying \eqref{eq:boundary conditions for momentum profile} is given by
	\begin{equation}\label{eq:final-sol-expression-of-phi-involving-a-and-b}
		\begin{split}
			(1 +\tau)\phi(\tau) &= -\frac{b^2 + 2a^2 b + 5a^2 + 4a}{a^2(a^2 + 6a + 6)(b+2)}\tau^4 \\ 
			&\quad  + \frac{(a^3 + 5a^2 + 6a -6)b^2 + 12a^3b + 12a(2a^2 - a -3)}{3a^2(a^2 + 6a + 6)(b+2)}\tau^3 \\ &\quad - \frac{-(a^2 + 5a +6)b^2 + a(a^2 - 6a - 6)b + 2(a^3 - 6a^2 - 3a + 6)}{a(a^2 + 6a + 6)(b + 2)}\tau^2 \\
			&\quad  + 2\tau + \frac{a^2 + 2a}{3(b^2 + 2b)} - \frac{a^2 + 2a}{3(b^2 + 2b)}\left(\frac{b^2 + 2b}{a^2 + 2a}(\tau^2 + 2\tau) + 1 \right)^{\frac{3}{2}}.
		\end{split}
	\end{equation}

	\subsection{Special case ($a = b$)}
	\label{subsec:special case a equals b}
	
	In this case, the K\"ahler classes $\Omega_a$ and $\Omega_b$ coincide. Taking $\chi = \om$, it is straightforward to see that $(\om, \chi)$ is a coupled extremal K\"ahler metric if and only if $\om$ is an extremal K\"ahler metric. Setting $a=b$, the unique solution $\phi$ given in \eqref{eq:final-sol-expression-of-phi-involving-a-and-b} simplifies to:
	\begin{align*}
		(1+\tau)\phi(\tau) = - \frac{2(a+1)}{a(a^2 + 6a + 6)}\tau^4 + \frac{3a^2 - 2a - 6}{a(a^2 + 6a + 6)}\tau^3 + \frac{- a^3 + 3 a^2 - 6}{a(a^2 + 6a + 6)}\tau^2 + \tau 
	\end{align*}
	Factoring out $\tau(a - \tau)$, this expression further reduces to (cf. \cite[p.~75]{Sze-bk}):
	\begin{equation}
		\phi(\tau) = \frac{\tau (a - \tau)}{a(a^2 + 6a + 6)(1+\tau)}\Big((2a + 2)\tau^2 + ( -a^2 + 4a + 6)\tau + (a^2 + 6a + 6)\Big).
	\end{equation}
	It is well known that $\phi(\tau) > 0$ for all $\tau \in (0,a)$ if and only if $0 < a < k_1$, where $k_1 \approx 18.889$ is the unique positive root of the quartic polynomial $\wp(x)$ given in \eqref{eq:quartic-polynomial-std-extremal-metric}. Consequently, the K\"ahler class $\Omega_a$ on the ruled surface $X$ admits an extremal K\"ahler metric with non-constant scalar curvature if and only if $0< a < k_1$.

	\vspace*{1.5mm}
	Furthermore, $X$ is not \emph{relatively K-stable} when $a \geq k_1$; hence, the K\"ahler class $\Omega_a$ does not admit an extremal K\"ahler metric in this range (cf. \cite[Section~6.5]{Sze-bk}; see also \cite{ACGTF2008-invent}). In~\cite{Sze2009-Cal-func-ruled}, Sz\'ekelyhidi studied the minimization problem of the Calabi functional on $\Omega_a$ for $a \geq k_1$. He proved that a minimizing sequence of metrics breaks $X$ into two pieces when $k_1 \leq a \leq k_{2}(k_{2} + 2)$, and into three pieces when $a > k_{2}(k_{2} + 2)$. Here, $k_{2}\approx 5.0275$ is the only positive real root of the cubic polynomial $x^3 - 3x^2 - 9x - 6$.

	\subsection{General Coupled Case ($a \neq b$)}
	
	For brevity, let $F(\tau) := (1+\tau)\phi(\tau)$. Then, from \eqref{eq:final-sol-expression-of-phi-involving-a-and-b}, the solution function $F$ takes the form
	\begin{equation}\label{eq:solution-function-F}
		F(\tau) = P_4(\tau) + \frac{a^2 + 2a}{3(b^2 + 2b)} \left( 1 - \left(\frac{b^2 + 2b}{a^2 + 2a}(\tau^2 + 2\tau) + 1\right)^{3/2} \right),
	\end{equation}
	where $P_4(\tau) := -A_4 \tau^4 + A_3 \tau^3 - A_2 \tau^2 + 2\tau$ is a fourth-degree polynomial with coefficients
	\begin{align*}
		A_4 &:= \frac{b^2 + 2a^2b + 5a^2 + 4a}{a^2(a^2 + 6a + 6)(b + 2)}, \\
		A_3 &:= \frac{(a^3 + 5a^2 + 6a - 6)b^2 + 12a^3b + 12a(2a^2 - a - 3)}{3a^2(a^2 + 6a + 6)(b + 2)}, \\
		A_2 &:= \frac{-(a^2 + 5a + 6)b^2 + a(a^2 - 6a - 6)b + 2(a^3 - 6a^2 - 3a + 6)}{a(a^2 + 6a + 6)(b + 2)}.
	\end{align*}
	Because $\phi$ satisfies the boundary conditions \eqref{eq:boundary conditions for momentum profile}, $F$ satisfies
	\begin{equation}\label{eq:bdd-cond-for-solution-function-F}
		F(0) = 0 = F(a), \quad F'(0)=1, \quad \text{and} \quad F'(a) = -(1 + a).
	\end{equation}
	These boundary conditions imply that $F(\tau) > 0$ both for sufficiently small $\tau > 0$ and for $\tau < a$ sufficiently close to $a$. By Rolle's theorem, if $F$ has a zero in $(0, a)$, then $F$ must have at least two critical points and $F'$ must have at least one critical point in $(0, a)$.

	\vspace*{1.5mm}
	To emphasize the dependence on parameters $a$ and $b$, we write $F(a, b, \tau)$ for $F(\tau)$. That is, for any $(x, y) \in \mathbb{R}_{>0}^2$, the function $F(x, y, \cdot) : [0, x] \to \mathbb{R}$ is defined by \eqref{eq:solution-function-F} with $a$ and $b$ replaced by $x$ and $y$, respectively. In particular, $F(x, y, \cdot)$ satisfies the following boundary conditions
	\begin{equation}\label{eq:bdd-cond-for-F-general-x-y}
		F(x, y, 0) = F(x, y, x) = 0, \qquad  \frac{\partial F}{\partial\tau}(x, y, 0) = - \frac{1}{1 + x}\frac{\partial F}{\partial\tau}(x, y, x) = 1.
	\end{equation}
	We define the set
	\begin{equation}\label{eq:defn-of-sols-set-S}
		\mathcal{S}_{\mathrm{ext}} := \Big\{(x, y)\in\mathbb{R}_{>0}^2 :~ F(x, y, \tau) > 0 \quad \text{for all} \quad \tau\in (0, x)\Big\}.
	\end{equation} 
	
\medskip
	
The following properties hold for this set.
	
\begin{lem}\label{lem:sols-set-contains-diagonal-segment}
The set $\mathcal{S}_{\mathrm{ext}}$ contains the diagonal segment $\Delta_{k_1}$, defined in \eqref{eq:diagonal-segment-std-extremal-metric}.
\end{lem}
	
\begin{proof}
When $x = y$ (i.e., when the K\"ahler classes $\Omega_x$ and $\Omega_y$ coincide), the coupled extremal system reduces to the classical extremal K\"ahler metric problem on $X$. As established in Section~\ref{subsec:special case a equals b}, $F(x, x, \tau) > 0$ for all $\tau\in(0, x)$ if and only if $0 < x < k_1$. Consequently, the diagonal segment $\Delta_{k_1}$ is contained in $\mathcal{S}_{\mathrm{ext}}$. 
\end{proof}

\begin{lem}\label{lem:openness-of-the-sols-set-F}
The set $\mathcal{S}_{\mathrm{ext}}$ is an open subset of $\mathbb{R}_{>0}^2$.
\end{lem}
	
\begin{proof}
Let $(x_0, y_0) \in \mathcal{S}_{\mathrm{ext}}$ be fixed. By the momentum construction, for any $(x, y) \in \mathbb{R}_{>0}^2$, the function $F(x, y, \cdot): [0,x]\longrightarrow\mathbb{R}$ satisfies the boundary conditions \eqref{eq:bdd-cond-for-F-general-x-y}. Since $(x,y) \mapsto F(x,y,\tau)$ and the derivative of $F$ up to second order in $\tau$ depend smoothly on $(x, y, \tau)$, there exists a compact neighborhood $V \subset \mathbb{R}_{>0}^2$ of $(x_0, y_0)$, a radius $\delta_0 > 0$ such that $B_{\delta_0}(x_0, y_0) \subset V$, and a constant $M > 0$ such that for all $(x, y) \in V$, we have
$$x > \frac{x_0}{2} \quad \text{and} \quad \left|\frac{\partial^2 F}{\partial\tau^2}(x,y,\tau)\right| \leq M \quad \text{ for all } \tau \in (0, x).$$
Choose any $\epsilon$ such that $0< \epsilon < \min\left\{\frac{2}{M}, \frac{2 + x_0}{M}, \frac{x_0}{4}\right\}$. Applying Taylor's theorem with remainder about $\tau = 0$, we have
$$F(x, y, \tau) = \tau + \frac{1}{2}\frac{\partial^2 F}{\partial \tau^2}(x, y, \xi)\tau^2 \geq \tau\left(1 - \frac{M}{2}\tau\right) > 0$$
for all $(x, y) \in V$ and all $\tau \in (0, \epsilon]$. Similarly, applying Taylor's theorem about $\tau = x$, we get
$$F(x, y, \tau) = (1 + x)(x - \tau) + \frac{(\tau - x)^2}{2}\frac{\partial^2 F}{\partial \tau^2}(x, y, \xi') \geq (1 + x)(x - \tau) - \frac{M}{2}(x - \tau)^2 > 0$$
for all $(x, y) \in V$ and all $\tau \in [x - \epsilon, x)$. 

On the other hand, $F(x_0, y_0, \tau) > 0$ for all $\tau\in(0, x_0)$ (because $(x_0, y_0) \in \mathcal{S}_{\mathrm{ext}}$), and hence
$$m_0 := \min_{\tau \in [\epsilon, \, x_0 - \frac{\epsilon}{2}]} F(x_0, y_0, \tau) > 0.$$
Now, $F$ is uniformly continuous on the compact set $K := V \times [\epsilon, x_0 - \frac{\epsilon}{2}]$. Thus, there exists $\delta \in \left(0, \min\{\delta_0, \frac{\epsilon}{2}\}\right)$ such that for all $(x, y)$ satisfying $\|(x, y) - (x_0, y_0)\| < \delta$ (and hence $(x, y) \in V$), we have
$$|F(x, y, \tau) - F(x_0, y_0, \tau)| < \frac{m_0}{2} \quad \text{for all } \tau \in \left[\epsilon, x_0 - \frac{\epsilon}{2}\right],$$
which implies $F(x, y, \tau) > \frac{m_0}{2} > 0$ on this interval.
		
Furthermore, the condition $\vert x - x_0\vert < \delta < \frac{\epsilon}{2}$ implies $x - \epsilon < x_0 - \frac{\epsilon}{2} < x$, so the intervals $(0, \epsilon]$, $\left[\epsilon, x_0 - \frac{\epsilon}{2}\right]$, and $[x - \epsilon, x)$ together cover $(0, x)$. Consequently, for all $(x, y)\in B_{\delta}(x_0, y_0)$, we have $F(x, y, \tau) > 0$ for all $\tau \in (0, x)$. This concludes the proof that $\mathcal{S}_{\mathrm{ext}}$ is open in $\mathbb{R}_{>0}^2$
\end{proof}
	
\begin{rem}
By Lemma \ref{lem:openness-of-the-sols-set-F}, for any fixed $x \in (0, k_1)$, there exists an open interval $I_x := (y_{-}(x), y_{+}(x))$ containing $x$ such that $F(x,y,\tau) > 0$ on $(0, x)$ for all $y \in I_x$. 
\end{rem}

	\vspace*{1.5mm}
	We now prove our main Theorem \ref{thm:coupled-extremal-metric-exis-1}.
	
	\begin{proof}[\textbf{Proof of Theorem \ref{thm:coupled-extremal-metric-exis-1}}]
		Let $\Omega_a = 2\pi(\mathrm{C} + a D_\infty)$ and $\Omega_b = 2\pi(\mathrm{C} + b D_\infty)$ be K\"ahler classes on the ruled surface $X = \mathbb{P}(L \oplus \mathcal{O})$ with $a > 0$ and $b > 0$. We seek a pair of coupled extremal K\"ahler metrics $(\omega, \chi) \in \Omega_a \times \Omega_b$ satisfying the Calabi ansatz. Under the Calabi ansatz, the coupled extremal metric system \eqref{eq:Ricci curv equal, N=2} and \eqref{eq:coupled extremal second equ, N=2} reduces to a pair of ordinary differential equations \eqref{eq:ODE for all Ricci equal, N=2} and \eqref{eq:ODE coupled extremal, N=2}, respectively, involving the momentum profile $\phi(\tau)$ associated with $\omega$ and a function $\psi(\tau)$ related to the momentum profile of $\chi$. The unique solution $\psi$ of \eqref{eq:ODE for all Ricci equal, N=2} satisfying the boundary condition \eqref{eq:boundary conditions for psi} is given by (see \eqref{eq:unique-positive-solution-psi})
		$$\psi(\tau) = -1 + \sqrt{\frac{b^2+2b}{a^2+2a}(\tau^2+2\tau)+1}.$$
		Observe that $\psi$ is increasing and positive on $(0,a)$. We denote $F(\tau) := (1+\tau)\phi(\tau)$. Then, substituting $\psi(\tau)$ into the second-order ODE \eqref{eq:ODE coupled extremal, N=2} and solving it with the boundary conditions \eqref{eq:bdd-cond-for-solution-function-F} yields (see \eqref{eq:solution-function-F})
		$$F(\tau) = P_4(\tau) + \frac{a^2+2a}{3(b^2+2b)}\left(1 - \left(\frac{b^2+2b}{a^2+2a}(\tau^2+2\tau)+1\right)^{3/2}\right).$$
		By construction, $(\omega, \chi)$ form a pair of coupled extremal K\"ahler metrics if and only if the function $\phi(\tau) = \frac{F(\tau)}{1+\tau}$ defines a valid momentum profile on $(0, a)$, which holds if and only if $F(\tau) > 0$ for all $\tau \in (0, a)$. Now, the set $\mathcal{S}_{\mathrm{ext}}$, defined in \eqref{eq:defn-of-sols-set-S}, is an open set in $\mathbb{R}_{>0}^2$ by Lemma \ref{lem:openness-of-the-sols-set-F}, and it contains the diagonal segment $\Delta_{k_1}$ by Lemma \ref{lem:sols-set-contains-diagonal-segment}. Thus, a pair $(\omega, \chi)$ of coupled extremal K\"ahler metrics satisfying the Calabi ansatz exists if and only if $(a, b) \in \mathcal{S}_{\mathrm{ext}}$.
	\end{proof}

	\vspace*{2mm}
	\section{Further analysis of the coupled extremal region}
	\label{sec:further-analysis-of-coupled-extrem-region}
	
	We have proved above that the coupled extremal existence region $\mathcal{S}_{\mathrm{ext}}$, defined in \eqref{eq:defn-of-sols-set-S}, is open in $\mathbb{R}^{2}_{> 0}$ and contains the diagonal segment $\Delta_{k_1}$. In this section, we aim to further analyze the structure of $\mathcal{S}_{\mathrm{ext}}$ and describe its boundary.

	\vspace*{1.5mm}
	For notational convenience, we set
	$$\eta(x,y) := \frac{y^2 + 2y}{x^2 + 2x}, \quad \text{and} \quad \zeta(x, y, \tau) := \sqrt{1 + \eta(x, y)(\tau^2 + 2\tau)}$$
	for $(x, y) \in \mathbb{R}^{2}_{> 0}$ and $\tau \in (0, x)$. Then the explicit formula \eqref{eq:solution-function-F} for $F$ (with $a$ and $b$ replaced by $x$ and $y$, respectively) can be expressed as
	\begin{equation}
		\label{eq:solution-F-shorter-formula}
		F(x, y, \tau) = P_4(x, y, \tau) + \frac{1-\zeta(x, y, \tau)^3}{3\eta(x, y)}.
	\end{equation}
	Recall that the point $(x, y)$ belongs to $\mathcal{S}_{\mathrm{ext}}$ if $F(x, y, \tau) > 0$ for all $\tau \in (0, x)$.

	\vspace*{1.5mm}
	The behaviour of $F$ for small $y > 0$ provides useful information about the lower boundary of $\mathcal{S}_{\mathrm{ext}}$. Clearly, $\eta(x, y)\to 0$, and hence $\zeta(x, y, \tau)\to 1$, as $y\to 0^{+}$. Moreover, we have
	$$\frac{\partial \zeta}{\partial y}(x, y, \tau) = \frac{\tau^2 + 2\tau}{2 \zeta(x, y, \tau)} \frac{\partial\eta}{\partial y}(x, y).$$
	Using L'H\^{o}pital's rule, we have
	$$\lim_{y\to 0^+} \frac{1 - \zeta(x, y, \tau)^3}{3 \eta(x, y)} = - \frac{\tau(\tau + 2)}{2}.$$
	On the other hand, a straightforward computation yields
	\begin{align*}
		\lim_{y\to 0^+} P_4(x, y, \tau) &= \frac{\tau}{2x(x^2 + 6x + 6)}\Big(-(5x + 4)\tau^3 + 4(2x^2 - x - 3)\tau^2 \\& \qquad\qquad - 2(x^3 - 6x^2 - 3x + 6)\tau + 4x(x^2 + 6x + 6)\Big).
	\end{align*}
	Combining these, from \eqref{eq:solution-F-shorter-formula}, we obtain
	\begin{equation}\label{eq:limit-of-F-when-y-goes-to-zero}
		\lim_{y\to 0^+} F(x, y, \tau) = \frac{\tau(x - \tau)}{2x(x^2 + 6x + 6)} G_x(\tau),
	\end{equation}
	where 
	$$G_x(\tau) : = (5x + 4)\tau^2 + (-3x^2 + 8x + 12)\tau + 2(x^2 + 6x + 6).$$
	The discriminant of the quadratic polynomial $G_x$ is (see \eqref{eq:quartic-polynomial-cor-non-exis})
	\begin{align*}
		\vartheta(x) = 9x^4-88x^3-280x^2-240x-48.
	\end{align*}
	Let $k_0 \approx 12.451$ be the unique positive root of the polynomial $\vartheta(x)$. At $x = k_0$, the quadratic $G_{k_0}$ has a double root
	\begin{equation*}
		\tau_{k_0} := \frac{3k_0^2 - 8k_0 - 12}{2(5k_0 + 4)}\approx 2.668.
	\end{equation*}

	\vspace*{1.5mm}
	For every $0< x < k_0$, we have $\vartheta(x) < 0$, which implies that $G_x(\tau) > 0$ for all $\tau\in\mathbb{R}$. Consequently, 
	$$\lim_{y\to 0^+} F(x, y, \tau) > 0$$ 
	for all $x \in (0, k_0)$ and all $\tau \in (0, x)$. Note, however, that for each fixed $x\in (0, k_0)$, this positive limit approaches zero as either $\tau\to 0^+$ or $\tau\to x^{-}$.

	\vspace*{1.5mm}
	\begin{proof}[\textbf{Proof of Corollary~\ref{cor:non-exis-coup-extrem-small-value-of-b}}]
		If $x > k_0$, then $\vartheta(x) > 0$. In this case, the quadratic polynomial $G_x$ has two distinct real roots in $(0, x)$, and $G_x$ is negative between them. In particular, for every $x > k_0$, there exists a $\tau_{x} \in (0, x)$ such that
		$$\lim_{y\to 0^+} F(x, y, \tau_x) < 0.$$
		By the continuity of $F$, this implies that $F(x, y, \tau_x) < 0$ whenever $y>0$ is sufficiently small. Consequently, for every $x > k_0$, $(x, y)\notin \mathcal{S}_{\mathrm{ext}}$ for all sufficiently small $y > 0$. Therefore, by Theorem~\ref{thm:coupled-extremal-metric-exis-1}, we conclude that for every $a > k_0$, there exists a $\delta_a > 0$ such that the pair of normalized K\"ahler classes $(\Omega_a, \Omega_b)$ does not admit a coupled extremal K\"ahler metric satisfying the Calabi ansatz whenever $b \in (0, \delta_a)$.
	\end{proof}

	\vspace*{1.5mm}
	We next investigate the opposite regime, namely the behaviour of $F$ when $y$ is large. For fixed $x>0$ and fixed $\tau\in(0,x)$, we obtain
	\begin{align*}
		\lim_{y\to\infty} \frac{\zeta(x, y, \tau)^3}{3y\,  \eta(x, y)} &= \frac{(\tau^2 + 2\tau)^{3/2}}{3 \sqrt{x^2 + 2x}}; \\
		\lim_{y\to\infty} \frac{P_4(x, y, \tau)}{y} &= -\frac{1}{x^2(x^2 + 6x + 6)}\tau^4 + \frac{x^3 + 5x^2 + 6x - 6}{3x^2(x^2 + 6x + 6)}\tau^3 + \frac{(x + 2)(x + 3)}{x(x^2 + 6x + 6)}\tau^2.
	\end{align*}
	Consequently, from \eqref{eq:solution-F-shorter-formula}, we have
	\begin{equation}\label{eq:limit-of-F-by-y-when-y-goes-to-infinity}
		\begin{split}
			\lim_{y\to\infty} \frac{F(x, y, \tau)}{y} &= -\frac{1}{x^2(x^2 + 6x + 6)}\tau^4 + \frac{x^3 + 5x^2 + 6x - 6}{3x^2(x^2 + 6x + 6)}\tau^3 \\ &\qquad \quad + \frac{(x + 2)(x + 3)}{x(x^2 + 6x + 6)}\tau^2 - \frac{(\tau^2 + 2\tau)^{3/2}}{3 \sqrt{x^2 + 2x}}.
		\end{split}
	\end{equation}
	The behavior of the right-hand side for $\tau\to0^+$ is particularly useful. Indeed,
	$$(\tau^2+2\tau)^{3/2} = 2\sqrt{2}\,\tau^{3/2} + O(\tau^{5/2}),$$
	and therefore
	\begin{equation}\label{eq:limit-of-F-by-y-when-y-goes-to-infinity-and-for-small-tau}
		\lim_{y\to\infty}\frac{F(x, y, \tau)}{y} = -\frac{2\sqrt{2}}{3\sqrt{x^2 + 2x}}\tau^{3/2} + O(\tau^2).
	\end{equation}

	\vspace*{1.5mm}
	\begin{proof}[\textbf{Proof of Proposition~\ref{prop:bddness-of-the-coupled-extremal-region}}]
		For a fixed $x > 0$, it follows from \eqref{eq:limit-of-F-by-y-when-y-goes-to-infinity-and-for-small-tau} that there exists a sufficiently small $\tau'_x \in (0, x)$ such that
		$$\lim_{y\to\infty}\frac{F(x, y, \tau'_x)}{y} < 0.$$
		Hence, there exists a constant $K(x) > 0$ such that $F(x, y, \tau'_x) < 0$ for all $y > K(x)$. Consequently, for every fixed $x > 0$, the set of values $y > 0$ for which $(x, y) \in \mathcal{S}_{\mathrm{ext}}$ is bounded above.
	\end{proof}

	\vspace*{1.5mm}
	We conclude with the following remarks. We expect the set $\mathcal{S}_{\mathrm{ext}}$ to be connected. As $(x, y)$ approaches the boundary $\partial \mathcal{S}_{\mathrm{ext}}$, the function $F(x, y, \tau)$ develops an interior local minimum that vanishes at some point $\tau_0 \in (0, x)$. Thus, for $(x, y) \in \partial \mathcal{S}_{\mathrm{ext}}$, the function $F(x, y, \cdot)$ satisfies \eqref{eq:bdd-cond-for-F-general-x-y},
	\begin{equation}\label{eq:boundary-tangency-cond}
		F(x, y, \tau_0) = 0 \quad \text{and} \quad  \frac{\partial F}{\partial\tau}(x, y, \tau_0) = 0,
	\end{equation}
	and $F(x, y, \tau) \geq 0$ for all $\tau\in[0, x]$. Setting $\phi_x (\tau) := \frac{F(x, y, \tau)}{1 + \tau}$ for $\tau\in[0, x]$, we see that $\phi_x$ satisfies
	\begin{align*}
		\phi_x(0) = \phi_x (\tau_0) = \phi_x(x) = 0, \quad \phi_{x}'(0) = - \phi_{x}'(x) = 1, \quad \text{and} \quad \phi_{x}'(\tau_0) = 0.
	\end{align*}
	By \cite[Definition 4]{Sze2009-Cal-func-ruled}, this implies that $\phi_x$ is a \emph{singular} momentum profile. Consequently, the metric $\omega_x$ associated with the momentum profile $\phi_x$ is \emph{degenerate}.

	\vspace*{1.5mm}
	Differentiating $\zeta(x, y, \tau)$ with respect to $\tau$, we obtain
	$$\frac{\partial\zeta}{\partial\tau}(x, y, \tau) = \frac{(1+\tau)\eta(x, y)}{\zeta(x, y, \tau)}.$$
	Consequently, from \eqref{eq:solution-F-shorter-formula}, it follows that
	\begin{align*}
		\frac{\partial F}{\partial\tau}(x, y, \tau) = \frac{\partial P_4}{\partial\tau}(x, y, \tau) - (1 + \tau) \zeta(x, y, \tau).
	\end{align*}
	Suppose there exists a $\tau_0 \in (0, x)$ satisfying \eqref{eq:boundary-tangency-cond}. Then, from \eqref{eq:solution-F-shorter-formula}, we have
	\begin{align*}
		\zeta(x, y, \tau_0)^3 = 1 + 3 \eta(x, y) P_4(x, y, \tau_0).
	\end{align*}
	Moreover, the above computation yields
	\begin{align*}
		\frac{\partial P_4}{\partial\tau}(x, y, \tau_0) = (1 + \tau_0) \zeta(x, y, \tau_0).
	\end{align*}
	It implies that
	\begin{align*}
		\left(\frac{\partial P_4}{\partial\tau}(x, y, \tau_0)\right)^3 = (1 + \tau_0)^3 \left(1 + 3 \eta(x, y) P_4(x, y, \tau_0)\right).
	\end{align*}
	We define
	\begin{align*}
		\Psi(x, y, \tau) := \left(\frac{\partial P_4}{\partial\tau}(x, y, \tau)\right)^3 - (1 + \tau)^3 \left(1 + 3 \eta(x, y) P_4(x, y, \tau)\right).
	\end{align*}
	Therefore, for every boundary point $(x, y) \in \partial\mathcal{S}_{\mathrm{ext}}$, there exists a $\tau_0 \in (0, x)$ such that
	\begin{equation}\label{eq:bdd-tangency-cond-in-terms-of-function-H}
		\Psi(x, y, \tau_0) = 0 \qquad \text{and} \qquad \frac{\partial P_4}{\partial\tau}(x, y, \tau_0) > 0.
	\end{equation}
	The second condition holds because $\zeta(x, y, \tau_0) > 0$ by definition.

	\vspace*{2mm}
	\section*{Acknowledgements}
	
	The author thanks Prof. Vamsi Pingali and Prof. Ved Datar for their insightful comments and valuable feedback on an earlier draft of this paper. This work was supported in part by an Institute Post Doctoral Fellowship from the Indian Institute of Technology Bombay. The author would also like to thank Prof. Saikat Mazumdar for hosting the postdoctoral stay.

	\vspace*{2mm}
	\section*{Declarations}
	
	\textbf{Data availability statement:} Data sharing not applicable to this article since no datasets were generated or analyzed during the current study.
	
	\vspace*{1.5mm}
	\textbf{Conflict of interest statement:} The author declare no conflict of interest.

	\vspace*{2.5mm}
	

\end{document}